\documentclass{amsart}
\usepackage[linkcolor=blue, citecolor = blue]{hyperref}
\hypersetup{colorlinks=true}

\usepackage{amsthm,amsfonts,amsmath,amssymb,amscd} 

\usepackage{verbatim}
\usepackage{mathtools}
\usepackage[color=orange!20, linecolor=orange]{todonotes}

\usepackage{tikz}
\usetikzlibrary{decorations.markings}
\tikzstyle{vertex}=[circle, draw, inner sep=0pt, minimum size=6pt]
\newcommand{\vertex}{\node[vertex]}

\usetikzlibrary{arrows,automata}

\newtheorem{thm}{Theorem}[section]
\newtheorem{prp}[thm]{Proposition}
\newtheorem{lm}[thm]{Lemma}
 
\theoremstyle{definition}
\newtheorem{crl}[thm]{Corollary}

\theoremstyle{remark}
\newtheorem{rmk}[thm]{Remark}

\newcommand{\sgn}{\mathrm{sgn}}
\newcommand{\gin}{\mathrm{in}}
\newcommand{\gout}{\mathrm{out}}

\newcommand{\sudda}[1]{}

\newcommand{\dx}{\partial}

\title{Path decompositions of digraphs and their applications to Weyl algebra}
\author{Askar Dzhumadil'daev, Damir Yeliussizov}
\date{}
\address{Kazakh-British Technical University, 59 Tole bi St, 050000, Almaty, Kazakhstan}
\email{dzhuma@hotmail.com, yeldamir@gmail.com} 

\begin{document}

\begin{abstract}
We consider decompositions of digraphs into edge-disjoint paths and describe their connection with the $n$-th Weyl algebra of differential operators. This approach gives a graph-theoretic combinatorial view of the normal ordering problem and helps to study skew-symmetric polynomials on certain subspaces of Weyl algebra. For instance, path decompositions can be used to study minimal polynomial identities on Weyl algebra, similar as Eulerian tours applicable for Amitsur--Levitzki theorem. We introduce the $G$-Stirling functions which enumerate decompositions by sources (and sinks) of paths.  
\end{abstract}

\maketitle

\section{Introduction}
Let $G$ be a digraph with possible multiple edges and loops. Suppose that edges of $G$ are labeled by distinct indices. We consider {\it decompositions of $G$ into edge-disjoint increasing paths}. This means that we partition the edge set into paths so that edge labels increase along every path. Paths that we consider are directed and not simple in general, i.e. they may contain cycles, but no repetition of any edge is allowed. Let us say that such decompositions are {\it principal}. Note that if~$G$~is decomposed into one path, it is clearly an Euler tour. If $G$ has one vertex $1$ and $m$ labeled loops $(1,1)$, then principal decompositions correspond to partitions of set $[m] := \{1, \ldots, m \}$.  

In our paper we consider applications of this combinatorial setting related to Weyl algebra. The main idea behind our results is connection of graph decompositions with differential operators. We introduce the {\it $G$-Stirling function} which counts decompositions by sources (and sinks) of paths and it is defined as follows:
$$
S_G(I) := \text{ the number of principal decompositions of $G$ with multiset of sources $I$.}
$$
If $G$ has one vertex, then $S_G(I)$ becomes Stirling number of the second kind $S(m,k)$, where $|I| = k$ and $G$ has $m$ labeled loops $(1,1)$. So, $S_G(I)$ is a path partition version of the classical Stirling numbers (of second kind). 

We obtain that coefficients in normal ordering composition of the $n$-th Weyl algebra $A_n$ generated by $x_1, \ldots, x_n, \dx_1, \ldots, \dx_n$ enumerate principal digraph decompositions with prescribed sets of sources and sinks. Related coefficients are the values of $G$-Stirling function. For example, the typical formula in our interpretation is 
$$
\prod_{\ell = 1}^m x_{i_{\ell}} \dx_{j_{\ell}} = \sum_{I} S_{G}(I) \prod_{i \in I} x_j \prod_{j \in J} \dx_j,
$$
where $(i_{\ell}, j_{\ell})$ is an edge of $G$ with the label $\ell$, sum runs over all multisets of sources $I$, and $J$ is a multiset of sinks (which is determined uniquely from the given $I$). 
This fact (in its general form, Theorem \ref{main}) gives a graph-theoretic combinatorial interpretation to the normal ordering problem, including the case $n = 1,$ which was studied well (e.g. \cite{blasiak1, blasiak, desouky2, katriel, mansour, mansour2, boson, varvak}). Apparently for $n = 1$, our interpretations are similar with graph combinatorial models studied in \cite{blasiak1}. 

Consider the skew-symmetric polynomials $s_m$ as $m$-ary operations on Weyl algebra
$$
s_m(X_1, \ldots, X_m) = \sum_{\sigma \in S_m} {\sgn}(\sigma) X_{\sigma(1)} \cdots X_{\sigma(m)}.
$$
We are interested in a question whether $s_m=0$ is an {\it identity} on a certain subspace  
$W$ of Weyl algebra ($W \subset A_n$), or whether it is an {\it  
$m$-commutator}, i.e. that
$s_m(X_1,\ldots,X_m)\in W$ for all $X_1,\ldots,X_m\in W.$ 

Weyl algebra has no polynomial identities except associativity (Corollary~5 of Theorem~1 in \cite{dzhum}). So, to explore possible nontrivial identities or commutators, one has to restrict the class to smaller subspaces. For example, a classical result due to  Lie, Jacobi, Poisson, is that the space 
$$
A_n^{(-,1)} := \langle u \dx_i\ |\ u \in K[x_1, \ldots, x_n], i=1,\ldots,n \rangle 
$$
can be identified as a space of vector fields $\text{Vect}(n)$ and it has a $2$-commutator,
$$[X,Y] = XY - YX \in A_n^{(-,1)}$$ 
for all $X, Y \in A_n^{(-,1)}$. In \cite{dzhum1} it was proved that $A_n^{(-,1)}$ has nontrivial $N$-commutator for $N = n^2 + 2n - 2$ and that $s_{N+2} = 0$ is identity. 

Note that the space $A_n^{(-,1)}$ can be endowed by a {\it left-symmetric multiplication}
$$
u \dx_i \circ v \dx_j = u\dx_i(v) \dx_j.
$$
Under this multiplication $A_n^{(-,1)}$ becomes a {\it left-symmetric algebra}, i.e. it satisfies the following identity
$$
(X, Y, Z) = (Y,X, Z), \text{ where } (X, Y, Z) = X \circ (Y \circ Z) - (X \circ Y) \circ Z.
$$
Left-symmetric algebras appear in  differential geometry and physics and they are known by many other names: Vinberg algebras, pre-Lie algebras, right-symmetric algebras, etc. 

In \cite{dzhum1} it was proved that the $N$-commutator $s_N$ is a well-defined operation not only under the associative multiplication, which says that for all $X_1,\ldots,X_N\in A_n^{(-,1)}$
$$s_N(X_1,\ldots,X_N)=
\sum_{\sigma\in S_n} X_{\sigma(1)}\cdot (\cdots (X_{\sigma(N-1)}\cdot X_{\sigma(N)})\cdots)
\in A_n^{(-,1)},$$
but it can also be presented as an $N$-commutator under the left-symmetric multiplication 
$$
 s_N^\circ(X_1,\ldots,X_N)=\sum_{\sigma\in S_n} X_{\sigma(1)}\circ (\cdots (X_{\sigma(N-1)}\circ X_{\sigma(N)})\cdots).$$

The next natural subspace of Weyl algebra is 
$$A_n^{(1,1)}:=\langle x_i\dx_j\ |\ 1\le i,j \le n \rangle.$$ 
Note that the space $A_n^{(1,1)}$ generates a subalgebra of $A_n^{(-,1)}$ as a {\it left-symmetric} algebra, 
$$X=x_i\dx_j,Y=x_s\dx_k\in A_n^{(1,1)}\Rightarrow X\circ Y=\delta_{j,s}x_i\dx_k\in A_n^{(1,1)},$$
but under the {\it associative} multiplication it is not closed,
$$X=x_i\dx_j,Y=x_s\dx_k\in A_n^{(1,1)}\Rightarrow X\circ Y=\delta_{j,s}x_i\dx_k+x_ix_s\dx_j\dx_k\notin A_n^{(1,1)}.$$
The famous Amitsur-Levitzki theorem \cite{al} states that 
 $s_{2n}=0$ is an identity of the {\it left-symmetric} algebra $A_n^{(1,1)}.$
 In \cite{dzhum2} it was proved that this identity can be prolonged to the identity of the whole left-symmetric algebra $A_n^{(-,1)}.$ 
 
Now a natural question arises about identities of $A_n^{(1,1)}$ as a subspace of the {\it associative} Weyl algebra. Numerical evidence shows that for $n = 1,2,3$ it behaves like Amitsur--Levitzki identity, i.e. $s_2 = 0, s_4 = 0, s_6 = 0$, respectively, are minimal polynomial identities on $A_n^{(1,1)}$. However, it turns out that this case is more complicated: $s_8$ is not an identity for $n = 4$. 

We study this problem using graph-theoretic approach. It is known that Amitsur--Levitzki theorem can be proved using Euler tours in digraphs \cite{bollobas, szigeti, swan} (or decompositions into one path in our case).  In fact, the normal ordering (or expansion) of polynomials $s_m$ has coefficients related to path decompositions (in some sense, generalized Euler tours). For instance, the coefficient at the first order term $x_{i}\dx_{j}$ in $s_{2n}$ is $0$, which reflects Amitsur--Levitzki theorem; it corresponds to the usual Euler tours. The next order coefficients ($x_{i_1} x_{i_2} \dx_{j_1} \dx_{j_2}$, etc.) index decompositions into two or more paths. Using this graph based scheme, we prove that $s_{2n}$ is not an identity on $A_n^{(1,1)}$ for $n > 3$. Note that the problem of finding the minimal polynomial identity on $A_n^{(1,1)}$ remains open, i.e. to find a minimal $c = c(n)$ for which $s_c= 0$ is identity. We  know its existence and the following bound: $2n < c \le n^2$ (for $n > 3$).

We also apply this technique to study the $N$-commutators on Weyl algebra. As we mentioned above, a space of differential operators of first order $A_{n}^{(-,1)}$  has a nontrivial $N$-commutator for $N=n^2+2n-2$ \cite{dzhum1} and a space of differential operators with one variable ($n = 1$) of order $p$  admits a nontrivial $N$-commutator for $N=2p$ \cite{dzhum}. In all these cases, $s_{N+1}=0$ is an identity. One can expect that this is a general situation: if $s_m=0$ is a minimal identity then in the pre-identity case $s_{m-1}$ gives a nontrivial $N$-commutator for $N=m-1.$  
Example of $A_n^{(1,1)}$ shows that this conjecture is not true. We prove that if an $N$-commutator on $A_n^{(1,1)}$ is nontrivial, then $N = 2.$

\section{Principal decompositions and $G$-Stirling functions} \label{std} 
We call decomposition of a digraph $G = (V, E)$ into $k$ edge-disjoint paths by $k$-decomposition. Let us suppose that edges of $G$ are {\it labeled} by $m$ indices, $E = \{e_1, \ldots, e_m \}$. We say that the {$k$-decomposition $E = P_1 \cup \ldots \cup P_{k}$} is {\it principal} if for every path $P_{i} = e_{\ell_1} \ldots e_{\ell_s}$ ($1 \le i \le k$) we have $\ell_1 < \ldots < \ell_s$. In other words, we decompose 
the edge set into several paths and the indices of edges {\it increase} along every path.  
For example, the graph $G_1$ with $E = \{(1,2), (2,1), (4,2), (1,4), (2,5), (4,3) \}$ and $V = \{1,2,3,4 \}$ has a principal $3$-decomposition ${e_1 e_2 e_4 \cup e_3 e_5 \cup e_6}$ (see Fig. 1 (a), (c)); ${e_1e_2 \cup e_4e_6 \cup e_3e_5}$ is also a principal decomposition, whereas ${e_1e_2 \cup e_4e_3 \cup e_5e_6}$ is not.

When $V = \{ 1\}$ and graph has $m$ labeled loop edges $(1, 1)$, the principal decompositions correspond to partitions of set $[m]$ into disjoint subsets. Further, we suppose that the digraph $G$ is presented by the vertex set $V = [n]$.

A {\it block} (or {\it $p$-block} if $p$ is specified) of a graph is a distinguished set of edges 
$\{e_1, \ldots, e_p \}.$ If graph is built up from several (disjoint) blocks, then we require that the edges of each block {must lie} in {\it distinct} paths. For example, the digraph $G_1'$ (see Fig.~1 (b)) built from three blocks $B_1 = \{e_1,e_2 \},$ $B_2 = \{e_3 \}$, $B_3 = \{e_4, e_5, e_6 \}$, has a principal 4-decomposition $e_1 e_5 \cup e_2 e_4 \cup e_3 \cup e_6$ (see Fig.~1 (d)). Note that a principal 3-decomposition of $G_1$, ${e_1 e_2 e_4 \cup e_3 e_5 \cup e_6}$, cannot be used for $G_1'$ since $e_1, e_2$ are from one block $B_1$ and thus cannot be in the same path.

\begin{center}
\[\begin{tikzpicture}[x=0.7cm, y=0.7cm,>=latex, thick]
	\vertex (1) at (0,2) [label=above left:$1$]{};
	\vertex (2) at (2.5,2) [label=above right:$2$]{};	
	\vertex (3) at (2.5,0) [label=below right:$3$]{};
	\vertex (4) at (0,0) [label=below left:$4$]{};
	
	\draw[->] (1) edge[bend left=20] node[above]{$e_1$} (2);
	\draw[->] (2) edge[bend left=20] node[below]{$e_2$} (1);
	\draw[->] (4) edge node[left]{$e_3$} (2);
	\draw[->] (1) edge node[left]{$e_4$} (4);
	\draw[->] (2) edge node[right]{$e_5$} (3);
	\draw[->] (4) edge node[below]{$e_6$} (3);

	\vertex (v1) at (6,2) [label=above left:$1$]{};
	\vertex (v2) at (8.5,2) [label=above right:$2$]{};	
	\vertex (v3) at (8.5,0) [label=below right:$3$]{};
	\vertex (v4) at (6,0) [label=below left:$4$]{};
	
	\draw[->] (v1) edge[red, bend left=20] node[black, above]{$e_1$} (v2);
	\draw[->] (v2) edge[red, bend left=20] node[black, below]{$e_2$} (v1);
	\draw[->] (v4) edge [black] node[black, left]{$e_3$} (v2);
	\draw[->] (v1) edge [blue] node[black, left]{$e_4$} (v4);
	\draw[->] (v2) edge [blue] node[black, right]{$e_5$} (v3);
	\draw[->] (v4) edge [blue] node[black, below]{$e_6$} (v3);
\end{tikzpicture}\]

\begin{tabular}{cc}
(a) Graph $G_1$ without blocks & (b) Graph $G_1'$ with edges\\
						 &  divided into three blocks\\
						 &  $B_1 = \{e_1,e_2 \}$ (red), \\
						 &  $B_2 = \{e_3 \}$ (black)\\
						 & $B_3 = \{e_4, e_5, e_6 \}$ (blue)
\end{tabular}

\[\begin{tikzpicture}[x=0.7cm, y=0.7cm,>=latex, thick]
	\vertex[fill] (p1) at (5.0,2) [label=above left:$1$]{};
	\vertex (p2) at (7.5,2) [label=above right:$2$]{};	
	\vertex[fill] (p4) at (5.0,0) [label=below left:$4$]{};
	\draw[->] (p1) edge[bend left=20] node[above]{$e_1$} (p2);
	\draw[->] (p2) edge[bend left=20] node[below]{$e_2$} (p1);
	\draw[->] (p1) edge node[left]{$e_4$} (p4);

	\vertex (pp2) at (11.5,2) [label=above right:$2$]{};	
	\vertex[fill] (pp3) at (11.5,0) [label=below right:$3$]{};
	\vertex[fill] (pp4) at (9.0,0) [label=below left:$4$]{};
	\draw[->] (pp4) edge node[left]{$e_3$} (pp2);
	\draw[->] (pp2) edge node[right]{$e_5$} (pp3);

	\vertex[fill] (ppp3) at (15.5,0) [label=below right:$3$]{};
	\vertex[fill] (ppp4) at (13.0,0) [label=below left:$4$]{};
	\draw[->] (ppp4) edge node[below]{$e_6$} (ppp3);
\end{tikzpicture}\]

(c) Graph $G_1$ is decomposed into three paths $e_1 e_2 e_4,$  $e_3 e_5$ and $e_6$. 

(Sources and sinks are shown black.)

\[\begin{tikzpicture}[x=0.7cm, y=0.7cm,>=latex, thick]

	\vertex [fill](vv1) at (6,2) [label=above left:$1$]{};
	\vertex (vv2) at (8.5,2) [label=above right:$2$]{};	
	\vertex [fill](vv3) at (8.5,0) [label=below right:$3$]{};
		
	\draw[->] (vv1) edge[red, bend left=20] node[black, above]{$e_1$} (vv2);
	\draw[->] (vv2) edge [blue] node[black, right]{$e_5$} (vv3);
	
	
	\vertex (vvv1) at (11,2) [label=above left:$1$]{};
	\vertex [fill](vvv2) at (13.5,2) [label=above right:$2$]{};	
	\vertex [fill](vvv4) at (11,0) [label=below left:$4$]{};

	\draw[->] (vvv2) edge[red, bend left=20] node[black, below]{$e_2$} (vvv1);
	\draw[->] (vvv1) edge [blue] node[black, left]{$e_4$} (vvv4);

	\vertex [fill](vvvv2) at (17.5,2) [label=above right:$2$]{};	
	\vertex [fill](vvvv4) at (15,0) [label=below left:$4$]{};
	\draw[->] (vvvv4) edge [black] node[black, left]{$e_3$} (vvvv2);
	

	\vertex [fill](vvvvv3) at (21.5,0) [label=below right:$3$]{};
	\vertex [fill](vvvvv4) at (19,0) [label=below left:$4$]{};
	\draw[->] (vvvvv4) edge [blue] node[black, below]{$e_6$} (vvvvv3);
\end{tikzpicture}\]

(d) Graph $G_1'$ is decomposed into four paths $e_1 e_5,$  $e_2 e_4$, $e_3$ and $e_6$. 

\

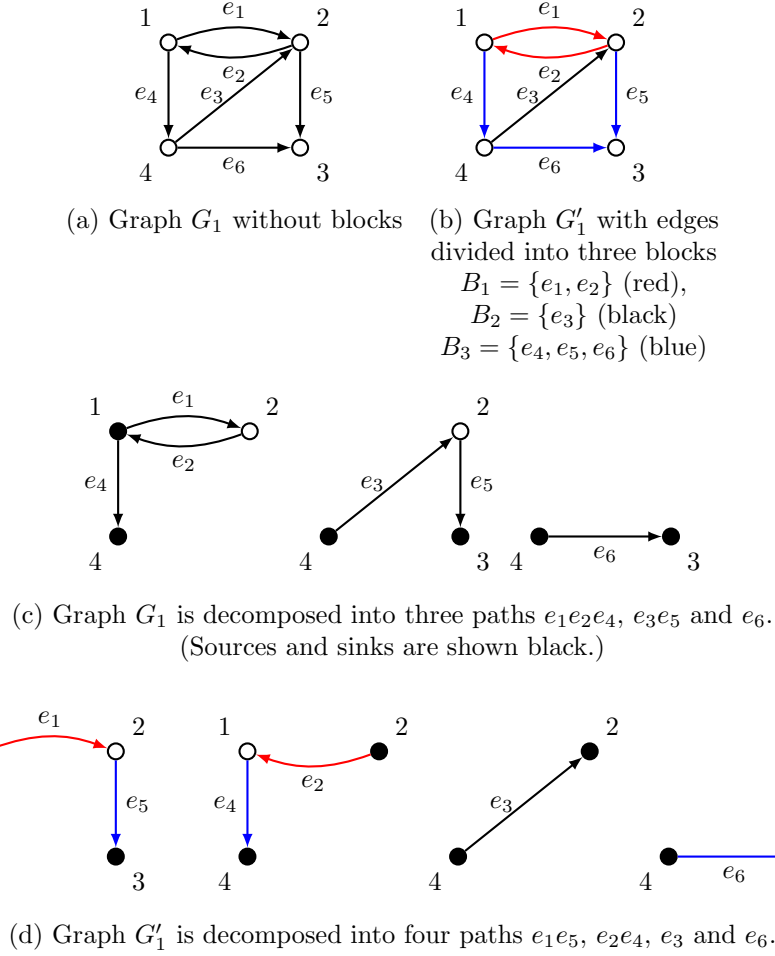
{\bf Fig. 1.} Examples of digraphs and principal decompositions.  
\end{center}

\

We will use the following notation for multisets: 
$A - X$ is a difference eliminating from $A$ as many copies of elements as $X$ has, e.g. $\{1^3,2^2,3,4^3 \} - \{1^2,2,4 \} = \{1,2,3,4^2 \}$; $A \uplus X$ is a merge of multisets, e.g. $\{1^2,2,4^2 \} \uplus \{1,2^2,3 \} = \{1^3,2^3,3,4^2 \}$. We also write $G - e$ if edge $e$ is eliminated from $G$ or $G - B$ if block $B$ is removed. 

For a given digraph $G$, let ${\gin(i)}, {\gout(i)}$ ($i \in V$) denote the number incoming and outcoming edges, respectively;
$$
V_{\gout} := \{1^{\gout(1)}, \ldots, n^{\gout(n)} \}, \quad V_{\gin}:=\{1^{\gin(1)}, \ldots, n^{\gin(n)} \},
$$
$$
M_{\gout}(G) := \{I\ | \ I \subseteq V_{\gout} \}, 
$$
i.e., $M_{\gout}$ is the set of all sub(multi)sets of $V_{\gout}.$ 

Note that if for a $k$-decomposition, we have the sources $I$, then the corresponding sinks $J = V_{\gin} \uplus I - V_{\gout}$ are determined uniquely. For example, in Fig. 1 (c) we have $I = \{ 1,4,4\}$ and $J = \{3,3,4 \}$. (Further, for any sources $I$ we will just write sinks as $J$ meaning that $J = V_{\gin} \uplus I - V_{\gout}$.)

Define {\it the $G$-Stirling function} 
$$
S_G : M_{\gout}(G) \to \mathbb{Z}_{\ge 0}
$$
as follows
$$
S_G(I) := {\text{ the number of principal decompositions of $G$}} \atop {\text{\quad\quad\quad with sources $I$ (and sinks $J$).}} 
$$

\

If $n = 1,$ then $S_G(I)$ corresponds to Stirling number of the second kind $S(m,k)$ where $|I| = k$ and digraph $G$ has $m$ labeled loops $(1,1)$. 

\begin{thm}\label{sg}
The $G$-Stirling function $S_G$ satisfies the following properties:
\begin{itemize}
\item[(i)] $S_G(V_{\gout}) = 1$; 
\item[(ii)] if $S_G(I) > 0$ for some $I \subset V_{\gout}$, then for any $I'$, such that $V_{\gout} \supseteq I' \supset I$, we have $S_G(I') > 0$;
\item[(iii)] Suppose that digraph $G$ is built up from blocks $B_1, \ldots, B_m$ so that the indices of edges increase with respect to the order of blocks. Let $e=(i,j) \in B_m$, $G' = G - e,$ $I' = I - \{ i\}$. Let $k_i$ be the number of repetitions of $i$ in $(J - \{ j\}) \uplus \{i \}$ and $r_{e}$ be the number of edges in $B_m - e$ that end by $i$. 
Then the following recurrence relation holds for $S_G(I)$.
\begin{align} \label{recmain}
S_{G}(I) &=S_{G'} (I') + (k_i - r_e) S_{G'}(I).
\end{align}

\end{itemize}
\end{thm}

\begin{proof}
The item (i) is clear, it corresponds to one principal $|E|$-decomposition of $G$. 

(ii) If there is a principal decomposition with sources $I$ then by additionally splitting certain paths at vertices $I' - I$ we may get a principal decomposition with sources $I'$. 

(iii) Note that if $S_G(I) > 0$, then $j \in J$.  
If edge $e$ forms a separate path in a principal decomposition of $G$, then we should have $i \in I,$ and the number of such decompositions is $S_{G'} (I').$ In the other cases, $e$ is the last edge of any path and can be joined by the vertex $i$ to decompositions of $G'$ having the same sources $I$ and sinks $(J - \{ j\}) \uplus \{i \}$ (by eliminating $e$ we remove $j$ and add $i$ to sinks). Since we cannot join $e$ after $r_e$ edges of the same block $B_m$, there are $(k_i - r_e)$ ways to join $e$ to every of $S_{G'}(I)$ corresponding decompositions. So, the recurrence follows. 
\end{proof}

\begin{crl}
If $G$ has edges $e_1, \ldots, e_m$ (without blocks), then for $e_m=(i,j)$, $G' = G - e_m,$ $I' = I - \{ i\}$ and $k_i$ the number of repetitions of $i$ in $(J - \{ j\}) \uplus \{i \}$, we have
\begin{align}\label{fgst}
S_{G}(I) &= S_{G'} (I') + k_i S_{G'}(I).
\end{align}
\end{crl}

\begin{crl}
If $G$ has $n = 1$ vertex and $m$ loops $(1,1)$, then $S_{G}(I) = S(m,k)$ if $|I| = k$, where $S(m,k)$ is Stirling number of the second kind. Relation \eqref{fgst} becomes the well-known recurrence
$$
S(m,k) = S(m-1, k-1) + k S(m-1, k).
$$
\end{crl}

\begin{rmk}
The $G$-Stirling function $S_G(I)$ is a graph (path-partition) generalization of Stirling number of the second kind; it counts partitions of (labeled) edge set into (increasing) paths. Note that $S_G$ is different from Stirling (and Bell) numbers for graphs studied in \cite{duncan}, which count partitions of graph vertex set into independent sets. Although, for $n = 1$ (and several blocks) there is a correspondence between these definitions (partitions of edge set into increasing paths vs. partitions of vertex set into independent sets).
\end{rmk}

\subsection{Symmetrization}

The symmetric group acts naturally on decompositions by permuting the indices of edges. For $\sigma \in S_m$ and digraph $G$ with the labeled edge set $E = \{e_1, \ldots, e_m\}$, let $G^{\sigma}$ be the same graph with edges labeled as $\{e_{\sigma(1)}, \ldots, e_{\sigma(m)}\}$. In general, this means that $G^{\sigma}$ will have another set of principal decompositions. 

Define the following characteristic
\begin{equation}\label{eg}
E_G(I) := \sum_{\sigma \in S_m} \sgn(\sigma) S_{G^{\sigma}}(I),
\end{equation}
where $I$ is any multiset on $[n]$. Note that if $|I| = 1,$ then $E_G$ reduces to the sum 
$$
E_{G}(\{i\}) = \sum_{e_{\sigma(1)} \cdots e_{\sigma(m)} \text{ Euler tours } i \to j} {\sgn(\sigma)},
$$
which has nice algebraic application \cite{bollobas, swan, szigeti} (here $j$ is the corresponding sink of an Euler tour). Namely, the following property is used in polynomial identities for matrix algebra: For a directed graph $G = (V, E)$ with $|V| = n$ and $|E| = 2n$ and every $1 \le i\le n$, we have $E_{G}(\{i\}) = 0.$
As we will see in next section, the characteristic $E_G(I)$ shows a similar connection with the Weyl algebra. 

We will also need the formula for computing $E_G(I)$ in terms of {\it shuffles} of paths,  which are defined as follows.

For permutations $\sigma, \tau$ of $\ell, r$ (disjoint) elements define the {\it shuffle set} 
$\text{Sh}(\sigma, \tau)$ as the set of all permutations of $\ell + r$ elements from $\sigma, \tau$ such that the order of elements from each of $\sigma$ and $\tau$ remains the same. For example, 
$$\text{Sh}((1,3), (4,2)) = \{(1,3,4,2), (1,4,3,2), (4,1,3,2), (1,4,2,3), (4,1,2,3), (4,2,1,3) \}.$$
For more than two permutations $\sigma_1, \ldots, \sigma_t$ the set $\text{Sh}(\sigma_1, \ldots, \sigma_t)$ is defined similarly. In other words, $\text{Sh}$ is the set of {\it linear extensions} of a poset that consists of separated chains labeled with respect to the given permutations.
Note that 
$$
|\text{Sh}(\sigma_1, \ldots, \sigma_t)| = \binom{|\sigma_1| + \cdots + |\sigma_t|}{|\sigma_1|, \ldots, |\sigma_t|} = \frac{(|\sigma_1| + \cdots + |\sigma_t|)!}{|\sigma_1|! \cdots |\sigma_t|!}.
$$

Consider now any $k$-decomposition (not necessarily principal) $\mathcal{P} = \{P_1, \ldots, P_k\}$ of $G$ with sources $I$ and sinks $J$; every path $P_i$ here is viewed as a permutation $({\ell_1}, \ldots, \ell_i)$ which presented by the sequence of edges $e_{\ell_1} \cdots e_{\ell_i}$. Define
\begin{align}
E(\mathcal{P}) &:= \sum_{\sigma \in \text{Sh}(P_1, \ldots, P_k)} {\sgn(\sigma)}.
\end{align}

\begin{prp} The following formula holds for $E_G(I)$,
\begin{equation}
E_G(I) = \sum_{\mathcal{P}: I \to J} {E(\mathcal{P})},
\end{equation}
where the sum is taken over all $k$-decompositions with sources $I$ and sinks $J$.
\end{prp}

\begin{proof}
Consider any permutation $\sigma \in S_m$. If we take a principal decomposition of $G^{\sigma}$ 
and apply $\sigma^{-1}$ to it, then we get a decomposition of $G$ with the same set of sources and sinks. Take any decomposition $\mathcal{P} = \{P_1, \ldots, P_k \}$ of $G$ and the set of permutations $\sigma$ for which $\sigma(\mathcal{P})$ becomes principal. Then for any path $P_i = e_{\ell_1} \ldots e_{\ell_s}$, we have $\sigma(\ell_1) < \cdots < \sigma(\ell_s)$. Therefore, for every $\sigma \in \mathrm{Sh}(P_1, \ldots, P_k)$, $\sigma^{-1}$ corresponds to a principal decomposition of $G^{\sigma}$. Note that $\sgn(\sigma) = \sgn(\sigma^{-1})$ and so we obtain
\begin{align*}
E_{G}(I) &= \sum_{\sigma \in S_m} \sgn(\sigma) S_{G^{\sigma}}(I) \\
&= \sum_{\mathcal{P}:I\to J} \sum_{\sigma \in S_m, \sigma(\mathcal{P}) \text{ is principal}} \sgn(\sigma)\\
&=\sum_{\mathcal{P}:I\to J} \sum_{\sigma \in \text{Sh}(P_1, \ldots, P_k)} {\sgn(\sigma)}.
\end{align*}
\end{proof}

\section{Connections with Weyl algebra}\label{norm-weyl}

\subsection{Definitions.} Let $K$ be a field of characteristic $0$. The $n$-th Weyl algebra $A_n$ is an associative algebra over $K$ defined by $2n$ generators $x_1, \ldots, x_n, \dx_1, \ldots, \dx_n$\footnote{$A_n$ is isomorphic to the polynomial algebra with $\dx_i$ considered as partial derivation $d/dx_i$.} and relations
$$
x_i x_j = x_j x_i, ~ \dx_i \dx_j = \dx_j \dx_i \text{, } \dx_i x_j  -  x_j \dx_i= \delta_{i,j} \text{ for } 1\le i,j \le n,
$$
where $\delta_{i,j}$ is the Kronecker symbol. The elements of types 
$$x^{\alpha} \dx^{\beta} := x_1^{\alpha_1} \cdots x_n^{\alpha_n} \dx_1^{\beta_1} \cdots \dx_n^{\beta_n}$$ 
with $\alpha = (\alpha_1, \ldots, \alpha_n), \beta = (\beta_1, \ldots, \beta_n) \in \mathbb{Z}^n_{\ge 0},$ are called {\it monomials}. Define the {\it length} 
$$\ell(x^{\alpha} \dx^{\beta}):= \sum_{i = 1}^n (\alpha_i - \beta_i)$$ and the {\it weight} $$\omega(x^{\alpha} \dx^{\beta}):= (\alpha_1 - \beta_1, \ldots, \alpha_n - \beta_n).$$
In most of the cases below, we will write monomials in the equivalent form 
$$x_{i_1} \ldots x_{i_s} \dx_{j_1} \ldots \dx_{j_p} \text{ for } i_1, \ldots, i_s, j_1, \ldots, j_p \in [n];$$
e.g. this monomial has length $s - p$. All monomials $x^{\alpha} \dx^{\beta}$ form a linear vector space basis of $A_n$. 
When the element $w$ of $A_n$ is expressed as a linear combination 
$$
w = \sum_{\alpha, \beta} c(\alpha, \beta) x^{\alpha} \dx^{\beta}, \quad c(\alpha, \beta) \in K,
$$
we say that $w$ is {\it normally ordered}. The {\it order} of $w$ is defined as 
$$\text{ord}(w) := \max_{c(\alpha, \beta) \ne 0}|\beta|, \quad |\beta| = \sum_{i = 1}^n \beta_i.$$ Note that $\text{ord}(w_1 w_2) = \text{ord}(w_1) + \text{ord}(w_2).$

Define the following subspaces of $A_n$: 
$$
A_{n}^{(p,q)} := \langle x^{\alpha} \dx^{\beta}\ :\ |\alpha| = p, |\beta| = q \rangle,
$$
$$
A_{n}^{(0)} := \bigoplus_{i \ge 1} A_{n}^{(i,i)}, \quad A_{n}^{*(p)} := \bigoplus_{i = 1}^p A_{n}^{(i,i)}.
$$
Note that $A_{n}^{(0)}$ is the subalgebra of $A_n$ formed by the elements of length $0$. 

\subsection{Normal ordering} We show that combinatorial meaning of coefficients in the normal ordering can be interpreted in terms of graph decompositions. Furthermore, we will consider monomials of subspace $A_{n}^{(0)},$ i.e. of length $0$ (otherwise, for our purposes we may add fictive elements). We associate every monomial $w = x_{i_1} \ldots x_{i_p} \dx_{j_1} \ldots \dx_{j_p} \in A_{n}^{(0)}$ with the $p$-block of a graph in the following way:
$$
\text{block}(w) := \{(i_1, j_1), \ldots, (i_p, j_p) \}.
$$ 

\begin{thm}\label{main}
Let $w_1,\ldots, w_m \in A_{n}^{(0)}$ be monomials. Then we have
\begin{equation}\label{wmain}
w_1 \cdots w_m = \sum_{I  \subseteq V_{\gout}} S_G(I) \prod_{i \in I} x_i \prod_{j \in J} \dx_j,
\end{equation}
where digraph $G$ with $n$ vertices is built up from the blocks $\text{block}(w_1), \ldots, \text{block}(w_m)$ (so that the indices of edges increase with respect to the order of blocks) and $J = V_{\gin} \uplus I - V_{\gout}$.
\end{thm}

Let us consider examples. 

\noindent{\bf Example 1.} Let $n = 4$ and 
$$w_1 = x_1x_2 \dx_2 \dx_1, w_2 = x_4 \dx_2, w_3 = x_1x_2x_4 \dx_4 \dx_3 \dx_3.$$
We have
$$
w_1 w_2 w_3 = 2 x_1 x_2 x_4^2 \dx_2 \dx_3^2 \dx_4 + x_1 x_2^2 x_4^2 \dx_2^2 \dx_3^2 \dx_4 + 2 x_1^2 x_2 x_4^2 \dx_1 \dx_2 \dx_3^2 \dx_4 + x_1^2 x_2^2 x_4^2 \dx_1 \dx_2^2 \dx_3^2 \dx_4
$$
and according to Theorem \ref{main}, digraph with $n = 4$ vertices is built up from three blocks $B_1 = \{e_1 = (1,2), e_2 = (2,1) \},$ $B_2 = \{e_3 = (4,2) \},$ $B_3 = \{e_4 = (1,4), e_5 = (2,3), e_6 = (4,3) \}.$ So, it is exactly the digraph shown in Fig. 1 (b). Table 1 shows its all principal decompositions and one can easily check that it corresponds to the expression above. 

\

\begin{center}
\begin{tabular}{ccc}
$I$ & $J$ & principal decompositions\\
\hline
\hline
$\{1,2,4,4 \}$ & $\{2,3,3,4 \}$ & $e_1 \cup e_2e_4 \cup e_3e_5 \cup e_6$\\
		& & $e_1 e_5 \cup e_2e_4 \cup e_3 \cup e_6$\\
\hline
$\{1,2,2,4,4 \}$  & $\{2,2,3,3,4 \}$ & $e_1 \cup e_2e_4 \cup e_3 \cup e_5 \cup e_6$ \\
\hline
$\{1,1,2,4,4 \}$ & $\{1,2,3,3,4 \}$ & $e_1 e_5 \cup e_2 \cup e_3 \cup e_4 \cup e_6$\\
& & $e_1 \cup e_2 \cup e_3e_5 \cup e_4 \cup e_6$ \\
\hline
$\{1,1,2,2,4,4 \}$ & $\{1,2,2,3,3,4\}$ & $e_1 \cup e_2 \cup e_3 \cup e_4 \cup e_5 \cup e_6$.
\end{tabular}

\

{\bf Table 1.} All principal decompositions of digraph $G_1'$ shown in Fig. 1 (b).
\end{center}

\

\noindent{\bf Example 2.} Suppose now $n = 3$ and 
$$
w_1 = x_1 \dx_1, w_2 = x_2 \dx_3, w_3 = x_2 \dx_1, w_4 = x_4 \dx_4, w_5  = x_1 \dx_2.
$$
We show how the expression 
$$
w_1 w_2 w_3 w_4 w_5  = 2 x_1 x_2^2 \dx_1 \dx_2 \dx_3 + 2x_1 x_2^2 x_3 \dx_1 \dx_2 \dx_3^2 + x_1^2 x_2^2 \dx_1^2 \dx_2 \dx_3 + x_1^2 x_2^2 x_3 \dx_1^2 \dx_2 \dx_3^2
$$
is related to graphs. According to Theorem \ref{main}, graph $G$ (see Fig. 2) consists of $n = 3$ vertices and edges $\{e_1 = (1,1), e_2 = (2,3), e_3 = (2,1), e_4 = (3,3), e_5 = (1,2) \}$. Table 2 shows all possible sources and sinks $I, J \subseteq \{1,2,3 \}$. Recall that $S_G(I)$ is the number of principal decompositions with sources $I$ and sinks $J$. For instance, we have two possible principal decompositions with 
$$I = \{1,2,2 \}, J = \{1,2,3 \}: e_1e_5 \cup e_2e_4 \cup e_3 \text{ and } e_1 \cup e_2e_4 \cup e_3e_5.$$ 
Therefore, $S_G(\{1,2,2 \}) = 2$, which contributes to the expression above as the summand $2 x_1 x_2^2 \dx_1 \dx_2 \dx_3$. 

\begin{center}
{\begin{tabular}{ccc}
$I$ & $J$ & principal decompositions\\
\hline
\hline
$\{1,2,2 \}$ & $\{1,2,3 \}$ & $e_1e_5\cup e_2e_4\cup e_3$\\
& & $e_1\cup e_2e_4\cup e_3e_5$\\
\hline
$\{ 1,2,2,3\}$ &  $\{1,2,3,3 \}$ & $e_1\cup e_2\cup e_3e_5\cup e_4$\\
& & $e_1e_5\cup e_2\cup e_3\cup e_4$\\
\hline
$\{1,1,2,2 \}$ & $\{1,1,2,3 \}$ & $e_1\cup e_2e_4\cup e_3\cup e_5$\\
\hline
$\{1,1,2,2,3 \}$ & $\{1,1,2,3,3 \}$ & $e_1\cup e_2\cup e_3\cup e_4\cup e_5$\\
\end{tabular}}

\

{\bf Table 2.} All principal decompositions of the graph presented in Fig. 2. 
\end{center}

\begin{center}
\begin{tikzpicture}[x=0.7cm, y=0.7cm,>=latex, thick]
	\vertex (1) at (0,2) [label=left:$1$]{};
	\vertex (2) at (3,2) [label=right:$2$]{};	
	\vertex (3) at (1.5,0) [label=right:$3$]{};
	\draw[->] (1) edge [black, loop above] node[black,above]{$e_1$} (1);
	\draw[->] (2) edge [black] node[black,right]{$e_2$} (3);
	\draw[->] (2) edge[black,bend left=20]  node[black,below]{$e_3$} (1);
	\draw[->] (3) edge [black,loop below] node[black,below]{$e_4$} (3);
	\draw[->] (1) edge[bend left=20]  node[above]{$e_5$} (2);
\end{tikzpicture}

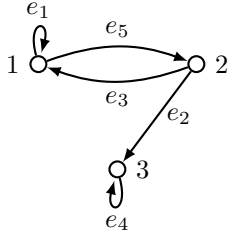
{\bf Fig. 2.}
\end{center}

\begin{proof}[Proof of Theorem \ref{main}]
We proceed by induction on the total order of $w_1, \ldots, w_m$, i.e. on the value
$$\text{ord}(w_1 \cdots w_m) = \text{ord}(w_1) + \cdots + \text{ord}(w_m).$$ The statement is obvious if the total order is $1$, or we have monomial $x_i \dx_j$. To prove the formula for monomials $w_1 \cdots w_m$, let $w_m = x_{i_1} \ldots x_{i_s} \dx_{j_1} \cdots \dx_{j_s}$ and consider the action of $w_1 \cdots w_{m-1}$ on $x_{i_s}$. For simplicity, put $i_s = i, j_s = j$ and $w_m' = x_{i_1} \cdots x_{i_{s-1}} \dx_{j_1} \cdots \dx_{j_{s-1}}$. Let $\{\dx_{i}^{(1)}, \ldots, \dx_{i}^{(q)} \}$ be all $\dx_{i}$'s in monomials $w_1, \ldots, w_{m-1}$. When one of $\dx_i$ acts on $x_i$, we will change this situation to the following equivalent operation: remove $x_i$, then change $\dx_i$ to the fictive element $\dx_{n+1}$, and after the normal ordering process remove $\dx_{n+1}$. Using this operation we obtain that
\begin{align*}
w_1 \cdots w_m = x_i (w_1 \cdots w_{m-1} w_m') \dx_j + \left(\sum_{\ell = 1}^{q} [w_1 \cdots w_{m-1}]_{\dx_i^{(\ell)} \to \dx_{n+1}} w_m' \right)_{\dx_{n+1} \to 1} \dx_j.
\end{align*}
(Here $[w_1 \cdots w_{m-1}]_{\dx_i^{(\ell)} \to \dx_{n+1}}$ means that we change $\dx_i^{(\ell)} \to \dx_{n+1}$ in one of $w_1, \ldots, w_{m-1}$.)
Note that we can apply the induction hypothesis to expressions $[w_1 \cdots w_{m-1}]_{\dx_i^{(\ell)} \to \dx_{n+1}} w_m'$ and $w_1 \cdots w_{m-1} w_m'$. Hence,
\begin{align*}
w_1 \cdots w_m &= x_i \left(\sum_{I} S_{G'}(I) \prod_{\ell \in I}x_{\ell} \prod_{k \in J}\dx_k \right) \dx_j \\
&+ \left(\sum_{\ell = 1}^q  \sum_{I} S_{G'_{\ell}}(I) \prod_{\ell \in I}x_{\ell} \prod_{k \in J}\dx_k \right)_{\dx_{n+1} \to 1} \dx_j,
\end{align*}
where graph with $n$ vertices $G'$ is built up from $\text{block}(w_1), \ldots, \text{block}(w_{m-1}), \text{block}(w_m')$; and graph $G'_{\ell}$ obtained by adding a new vertex $n+1$ and changing the edge $e = (v,i)$ that corresponds $\dx_{i}^{\ell}$ to $e := (v, n + 1).$ Note that $S_{G'_{\ell}}(I) = S_{G'}(I)$ (with sinks $(J - \{n + 1\}) \uplus \{ i \}$). Therefore, we get
\begin{align*}
w_1 \cdots w_m &= \sum_{I} S_{G'}(I - \{ i\}) \prod_{\ell \in I}x_{\ell} \prod_{k \in J}\dx_k \\
&+  \sum_{I} q S_{G'}(I) \prod_{\ell \in I}x_{\ell} \prod_{k \in J}\dx_k,\\
&=\sum_{I} (S_{G'}(I - \{ i\}) + q S_{G'}(I)\ \prod_{\ell \in I}x_{\ell} \prod_{k \in J}\dx_k\\
&=\sum_{I, J} S_{G}(I) \prod_{\ell \in I}x_{\ell} \prod_{k \in J}\dx_k,
\end{align*}
where $G$ is built up from $\text{block}(w_1), \ldots, \text{block}(w_m)$; we have used Theorem \ref{sg} (eq. \eqref{recmain}) for which it is easy to see that $q$ is a number of $i$'s in $(T - \{j\}) \uplus \{i \}$ without counting the last block.
\end{proof}

\begin{rmk}
In fact, the monomial $w = x_{i_1} \ldots x_{i_p} \dx_{j_1} \ldots \dx_{j_p}$ can be associated with any $p$-block of a graph that matches the vertices $i_1, \ldots, i_p$ with $j_1, \ldots, j_p$, e.g. for every permutation $\sigma \in S_p$ we may define
$$
\text{block}(w) = \{(i_1, j_{\sigma(1)}), \ldots, (i_p, j_{\sigma(p)}) \}.
$$
Note that these changes do not affect on the result of Theorem \ref{main}, the right-hand side remains the same.
\end{rmk}

\begin{crl} For a digraph $G = ([n], E)$ with $E = \{e_1, \ldots, e_m \}$, we have
\begin{equation}\label{e2}
\prod_{\ell = 1}^m x_{i_{\ell}} \dx_{j_{\ell}} = \sum_{I} S_G(I) \prod_{i \in I} x_i \prod_{j \in J} \dx_j,
\end{equation}
where $e_{\ell} = (i_{\ell}, j_{\ell})$, the sum runs over all (multi)sets of sources $I$, and $J$ is a set of sinks.
\end{crl}

\begin{crl}
If $n = 1,$ then \eqref{e2} becomes the classical result
$$
(x \dx)^m = \sum_{i = 0}^m S(m,i) x^i \dx^i,
$$
where $S(m,i)$ is Stirling number of the second kind.
\end{crl}

\subsection{Skew-symmetric polynomials}\label{skew-weyl}

Consider the skew-symmetric polynomial over noncommuting variables
$$
s_{n}(x_1, \ldots, x_n) := \sum_{\sigma \in S_{n}} {\sgn(\sigma)} x_{\sigma(1)} \cdots x_{\sigma(n)}. 
$$
The famous Amitsur-Levitzki theorem \cite{al} states that
$$
s_{2n}(A_1, \ldots, A_{2n}) = 0
$$
is a minimal polynomial identity for $n \times n$ matrices $A_1, \ldots, A_{2n}$.
This result is also known as an application of Euler tours to algebra \cite{bollobas, swan, szigeti}. Namely, if we have a digraph $G = (V, E)$ with $|V| = n$ and $|E| = 2n,$ then for every $1 \le i,j \le n$
$$
\sum_{e_{\sigma(1)} \cdots e_{\sigma(2n)} \text{ Euler tours } i \to j} {\sgn(\sigma)} = 0.
$$

We will now present a similar connection of graph theory with the Weyl algebra. 

Recall that the subspace $A_n^{(1,1)} \subset A_n$ is generated by monomials $x_i \dx_j$,
$$
A_n^{(1,1)} = \langle x_{i} \dx_{j} \ |\ i, j \in [n] \rangle.
$$
We show the following skew-symmetric analog of Theorem \ref{main}.

\begin{thm} \label{symw}
Let $w_1, \ldots, w_{m} \in A_n^{(1,1)}$ be monomials. Then
$$
s_{m}(w_1, \ldots, w_m) = \sum_{I} E_G(I) \prod_{i \in I} x_i \prod_{j \in J} \dx_j,
$$
where digraph $G$ with $n$ vertices has $m$ edges represented by $w_1, \ldots, w_m$ (i.e. if $w_{\ell} = x_{i_{\ell}} \dx_{j_{\ell}}$, then there is an edge $(i_{\ell}, j_{\ell})$ in $G$).
\end{thm}

\begin{proof}
From Theorem \ref{main}, 
$$
w_{\sigma(1)} \cdots w_{\sigma(m)} = \sum_{I} S_{G^{\sigma}}(I) \prod_{i \in I} x_i \prod_{j \in J} \dx_j,
$$
where $S_{G^{\sigma}}(I)$ enumerates principal decompositions with respect to the edges permutation $\sigma$. 
Therefore,
\begin{align*}
s_{m}(w_1, \ldots, w_m) &= \sum_{I} \sum_{\sigma \in S_m} \sgn(\sigma) S_{G^{\sigma}}(I, J) \prod_{i \in I} x_i \prod_{j \in J} \dx_j \\
&= \sum_{I} E_G(I) \prod_{i \in I} x_i \prod_{j \in J} \dx_j.
\end{align*}
\end{proof}

\begin{rmk}
Theorem \ref{symw} presents a normal ordering of the skew-symmetric expression. We will see that this form is useful in investigating the skew-symmetric identities. 
\end{rmk}

\subsection{Minimal polynomial identities} We say that $s_m$ is a {\it minimal polynomial identity} on some space $W$ if 
$$s_m(X_1, \ldots, X_m) = 0 \text{ for every } X_1, \ldots, X_m \in W$$ and 
$$s_{m-1}(X_1, \ldots, X_{m-1}) \ne 0 \text{ for some } X_1, \ldots, X_{m-1} \in W.$$

Amitsur-Levitzki theorem gives a hint that the coefficient of any order $1$ term $x_{i}\dx_{j}$ in $s_{2n}(w_1, \ldots, w_{2n})$ is 0 (it sums with a sign for all Euler tours from $i$ to $j$). In next theorem we show that the same is not always true for coefficients at other terms.

\begin{thm}\label{ssym} The following properties hold for $s_m$ on $A_{n}^{(1,1)}$.

\begin{itemize}
\item $s_{2n} = 0$ is a minimal identity on $A_n^{(1,1)}$ for $n = 1, 2, 3$. 
\item $s_{10} = 0$ is a minimal identity on $A_4^{(1,1)}$.
\item For $n > 3,$ $s_{2n}$ is not an identity on $A_n^{(1,1)}$.
\end{itemize}
\end{thm}

We first need the following result.

\begin{lm}\label{shuf}
Let 
$$
\mathrm{Sh}(m,n) := \mathrm{Sh}((1,\ldots, m), (m + 1, \ldots, m + n))
$$
and 
$$
q(m,n) := \sum_{\sigma \in \mathrm{Sh}(m,n)} {\sgn(\sigma)}.
$$
Then 
$$
q(m,n) = q(n,m), \quad q(2m-1, 2n-1) = 0,
$$
$$
q(2m,2n) = q(2m+1, 2n) = \binom{m + n}{n}.
$$
\end{lm}

\begin{proof}
By the definition, it is obvious that $q(m,n) = q(n,m)$. Let us compute the recurrence for $q(m,n)$. If the last element of permutation is $m + n,$ then we have the sum $q(m,n-1)$. Otherwise, the last element is $m$ which gives $(-1)^n q(m-1,n)$. Hence we have
$$
q(m,n) = q(m,n-1) + (-1)^n q(m-1,n).
$$ 
(In fact, $q(m,n)$ is a $q$-binomial coefficient at $q = -1$.)
So, the needed formulas can easily be derived by induction, since we have
\begin{align*}
q(2m+1,2n+1) &= q(2m+1,2n) - q(2m,2n+1) = \binom{m + n}{n} - \binom{m + n}{m} = 0,
\end{align*}
\begin{align*}
q(2m,2n) &= q(2m,2n-1) + q(2m-1,2n) = \binom{m + n - 1}{n-1} + \binom{m-1+n}{n} = \binom{m + n}{n},
\end{align*}
$$
q(2m+1,2n) = q(2m+1,2n-1) + q(2m,2n) = q(2m,2n).
$$
\end{proof}

\begin{proof}[Proof of Theorem \ref{ssym}]

First note that $s_{m}(w_1, \ldots, w_m) = 0$ if some of $w_1, \ldots, w_m$ are equal. 

1) $s_2 = 0$ is identity for $n = 1$. It is obvious that $s_2(x\dx, x\dx) = (x \dx)^2 - (x\dx)^2 = 0.$ 

$s_4 = 0$ is identity for $n = 2.$ Here we may consider only the case with four operators $x_1 \dx_1, x_2 \dx_2, x_1 \dx_2, x_2 \dx_1$. It can easily be checked that $s_{4}(x_1 \dx_1, x_2 \dx_2, x_1 \dx_2, x_2 \dx_1) = 0$. 

$s_6 = 0$ is identity for $n = 3$.  
There are 17 such cases up to symmetry; and all can easily be verified.

2) $s_{10} = 0$ is identity for $n = 4$. This is verified from our computer calculations for all the possible cases (with reductions up to symmetry). 

3) To prove that $s_{2n}$ is not identity for $n > 3$, we show that $2$-Euler paths of graphs $G$ defined in Fig. 3 $E_G(\{1,1 \})$ does not sum to $0$. The latter means from Proposition \ref{symw} that the coefficient of $x_1^2 \dx_1^2$ in $s_{2n}$ is nonzero.

Suppose $n$ is even. We look for all cases of decompositions of $G$ (see Fig. 3, left) with $I = \{1,1 \}, J = \{1,1 \}$. 
For every vertex $i$ ($2 \le i \le n$) consider the paths $e_1\cdots e_{i-1} e_{2n-i+2} \cdots e_{2n}$ and $e_{n+1}\cdots e_{2n-i+1} e_{i} \cdots e_{n}$. 
These permutations will sum to 
$(-1)^{n-i+1} |\text{Sh}(2(i-1), 2(n-i+1))|$, which by Lemma \ref{shuf} gives $q(2(i-1), 2(n-i+1)) = (-1)^{n-i+1} \binom{n}{i-1}$. There are two more paths $e_1 \cdots e_n$ and $e_{n+1} \cdots e_{2n}$, for which we have $q(n,n) = \binom{n}{n/2}$. Therefore,
\begin{align*}
E_G(\{1,1 \}) &= \sum_{i = 2}^n (-1)^{n-i+1} \binom{n}{i-1} + \binom{n}{n/2}\\
&= -1 - (-1)^n + \binom{n}{n/2} > -1 - (-1)^n + 2 \ge 0.
\end{align*}

\begin{center}
\begin{tikzpicture}[x=0.7cm,y=0.7cm,>=latex, thick]

\def \n {10}
\def \radius {3*0.7cm}
\def \margin {2} 

\vertex[fill] (1) at ({360/\n * (1 + 2)}:\radius) [label = above:$1$]{};
\vertex (2) at ({360/\n * (10 + 2)}:\radius) [label = above:$2$]{};
\vertex (3) at ({360/\n * (9 + 2)}:\radius) {};

\vertex (n) at ({360/\n * (2 + 2)}:\radius) [label = left:$n$]{};
\vertex (n1) at ({360/\n * (3 + 2)}:\radius) {};

\draw[->] (1) edge[bend left = 20]  node[above] {$e_1$} (2);
\draw[->] (2) edge[bend left = 20]  node[below] {$e_{2n}$} (1);
\draw[->] (2) edge[bend left = 20]  node[above] {$e_2$} (3);
\draw[->] (3) edge[bend left = 20]  node[below] {$e_{2n-1}$} (2);

\draw (n1) edge[dashed, bend right = 100]  node[above] {} (3);

\draw[->] (n1) edge[bend left = 20]  node[left] {$e_{n-1}$} (n);
\draw[->] (n) edge[bend left = 20]  node[right] {$e_{n+2}$} (n1);
\draw[->] (n) edge[bend left = 20]  node[left] {$e_n$} (1);
\draw[->] (1) edge[bend left = 20]  node[right] {$e_{n+1}$} (n);
\end{tikzpicture}
\begin{tikzpicture}[x=0.7cm,y=0.7cm,>=latex, thick]

\def \n {10}
\def \radius {3*0.7cm}
\def \margin {2} 

\vertex[fill] (1) at ({360/\n * (1 + 2)}:\radius) [label = above:$1$]{};
\vertex (2) at ({360/\n * (10 + 2)}:\radius) [label = above left:$2$]{};
\vertex (3) at ({360/\n * (9 + 2)}:\radius) {};

\vertex (q) at ({360/\n * (10 + 2)}:\radius + 2*0.7cm) [label = above:$n$]{};

\draw[->] (2) edge[bend left = 20]  node[above left] {$e_{2n-1}$} (q);
\draw[->] (q) edge[bend left = 20]  node[right] {$e_{2n}$} (2);

\vertex (n) at ({360/\n * (2 + 2)}:\radius) [label = left:$n-1$]{};
\vertex (n1) at ({360/\n * (3 + 2)}:\radius) {};

\draw[->] (1) edge[bend left = 20]  node[above] {$e_1$} (2);
\draw[->] (2) edge[bend left = 20]  node[below] {$e_{2n-2}$} (1);
\draw[->] (2) edge[bend left = 20]  node[above] {$e_2$} (3);
\draw[->] (3) edge[bend left = 20]  node[below] {$e_{2n-3}$} (2);

\draw (n1) edge[dashed, bend right = 100]  node[above] {} (3);

\draw[->] (n1) edge[bend left = 20]  node[left] {$e_{n-2}$} (n);
\draw[->] (n) edge[bend left = 20]  node[right] {$e_{n+1}$} (n1);
\draw[->] (n) edge[bend left = 20]  node[left] {$e_{n-1}$} (1);
\draw[->] (1) edge[bend left = 20]  node[right] {$e_{n}$} (n);
\end{tikzpicture}

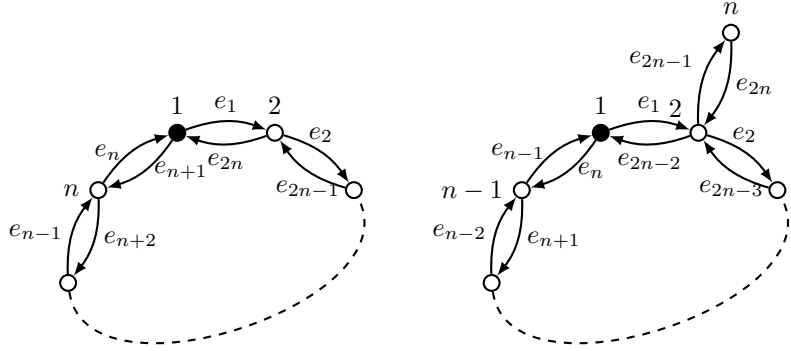
{\bf Fig. 3.} Graphs $G$ with $E_G(\{1,1 \}) \ne 0$ for $n$ even (left) and odd (right)
\end{center}

If now $n$ is odd, then we consider graph $G$ as in Fig. 3 (right). We again look for all decompositions with $I = \{1,1 \}, J = \{1,1 \}$. For every vertex $i$($3 \le i \le n-1$) we have the following two possibilities of paths: 
$$P_1 = e_n e_{n+1} \ldots e_{2n-1-i} e_i e_{i+1} \ldots e_{n-1}; P_2 = e_1 \ldots e_{i-1} e_{2n-i} \ldots e_{2n-3} e_{2n-1} e_{2n} e_{2n-2}$$
and 
$$P_1 = e_n e_{n+1} \ldots e_{2n-1-i} e_i e_{i+1} \ldots e_{n-1}; P_2 = e_1 e_{2n-1} e_{2n}  e_2 \ldots e_{i-1} e_{2n-i} \ldots e_{2n-2}.$$
For both cases we get the sum of $(-1)^{n-i} |\text{Sh}(2(n-i), 2i)|$, which is $(-1)^{n-i}\binom{n}{i}$. The remaining four cases of paths decompositions are
$$P_1 = e_1 e_{2n-2}; P_2 = e_n \ldots e_{2n-3} e_{2n-1} e_{2n} e_2 \ldots e_{n-1},$$
with sum of $-|\text{Sh}(2,2n-2)| = -\binom{n}{1}$;
$$P_1 = e_1 e_{2n-2}; P_2 = e_n \ldots e_{2n-3} e_{2n-1} e_{2n} e_2 \ldots e_{n-1},$$
with sum of $-|\text{Sh}(4,2n-4)| = -\binom{n}{2}$;
$$P_1 = e_1 \ldots e_{n-1}; P_2 = e_n \ldots e_{2n-3} e_{2n-1} e_{2n} e_{2n-2},$$
with sum of $|\text{Sh}(n-1, n+1)| = \binom{n}{(n-1)/2}$;
$$P_1 = e_1 e_{2n-1} e_{2n} e_2 \ldots e_{n-1}; P_2 = e_n \ldots e_{2n-2},$$
with sum of $|\text{Sh}(n+1, n-1)| = \binom{n}{(n-1)/2}$.
So, we obtain
\begin{align*}
E_{G}(\{1,1 \}) &= \left(2 \sum_{i = 3}^{n-1} (-1)^{n-i}\binom{n}{i}\right) - \binom{n}{1} - \binom{n}{2} + 2\binom{n}{(n-1)/2} \\
&= 2\left((1-1)^n - (-1)^n -(-1)^{n-1}\binom{n}{1} - (-1)^{n-2}\binom{n}{2} - 1 \right) \\
&- \binom{n}{1} - \binom{n}{2} + 2\binom{n}{(n-1)/2}\\
&=\binom{n}{2} - 3\binom{n}{1} + 2\binom{n}{(n-1)/2} \ge \binom{n}{2} - 3\binom{n}{1} + 2 \binom{n}{2} > 0.
\end{align*}
(Here $n \ge 5$.)
\end{proof}

\begin{rmk}
Reducing the non-identity case to computing $E_G(I)$ for some sources $I$ gives a more efficient way to analyze the sum instead of looking at the whole $s_m$. Computing $E_G(I)$ for all sources $I$ is apparently faster than computing $s_m$ directly (which at least is evident in smaller cases computations).   
\end{rmk}

\begin{rmk}
From our computations, most likely that $s_{12}$ is a minimal identity on $A_{5}^{(1,1)}$.
In fact, one can reduce the number of cases in computations by proving the following equivalent properties:
\begin{itemize}
\item[(A)] Suppose there are monomials $X_1, \ldots, X_m \in A_n^{(1,1)}$ such that $s_m(X_1, \ldots, X_m) \ne 0$. Then there are monomials $X'_1, \ldots, X'_m \in A_n^{(1,1)}$ such that $s_m(X'_1, \ldots, X'_m) \ne 0$ and with total weight 0, i.e.
$$
\omega(X'_1) + \cdots + \omega(X'_m) = (0, \ldots, 0).
$$
\item[(B)] If $E_{G}(I) \ne 0$ for some multisets $I$ and digraph $G$, then there is a {\it balanced digraph} $G'$ (i.e. $\gin(v) = \gout(v)$ for each vertex $v$) with the same number of vertices and edges, such that $E_{G'}(I') \ne 0$ for some multiset $I'$.
\end{itemize}
\end{rmk}

\subsection{$N$-commutators} $s_N$ is called {\it $N$-commutator} on $A_n^{(p,p)}$ if $s_N(X_1, \ldots, X_N) \in A_n^{(p,p)}$ for every $X_1, \ldots, X_N \in A_n^{(p,p)}$. If $s_N(X_1, \ldots, X_N) \ne 0$ for some $X_1, \ldots, X_N \in A_n^{(p,p)}$, $N$-commutator is nontrivial. 

It is known that the space of differential operators of first order 
$A_{n}^{(-,1)} = \langle u \dx_i\ |\ u \in K[x_1, \ldots, x_n] \rangle$ has a nontrivial $N$-commutator for $N=n^2+2n-2$ \cite{dzhum1} and a space of differential operators with one variable ($n = 1$) of order $p$  admits a nontrivial $N$-commutator for $N=2p$ \cite{dzhum}, i.e. there is a nontrivial $2p$-commutator on the subspace $\langle u \dx^{p} : u \in K[x] \rangle$. In all these cases, $s_{N+1}=0$ is an identity. One can expect that this is a general situation: if $s_m=0$ is a minimal identity then in the pre-identity case $s_{m-1}$ gives a nontrivial $N$-commutator for $N=m-1.$  
In next theorem we show that this is not true for the subspace $A_n^{(1,1)}$.

\begin{thm}
Let $s_N$ be a nontrivial $N$-commutator on $A_n^{(1,1)}$. Then $N = 2.$
\end{thm}

\begin{proof}
Suppose $N > 2$. If $N \ge 2n$, then $\text{ord}(s_N) \ge 2$ since coefficients at terms $x_i \dx_j$ that are $E_G(i)$ vanish from the Amitsur-Levitzki theorem. This means that $s_N \not\in A_n^{(1,1)}$. 

For the other cases, we adopt the graph-theoretic version of example used in proof of Amitsur-Levitzki theorem (that $s_{2n-1}$ is nonzero).

If $N < 2n,$ let us choose the first $X_1, \ldots, X_N$ operators from the set (of $2n-1$)
$$
x_1 \dx_1, x_1 \dx_2, x_2\dx_2, \ldots, x_{n-1} \dx_{n}, x_{n} \dx_n.
$$
The latter represents the graph $G$ with edges $(1,1), (1,2), (2,2), \ldots, (n-1,n), (n,n)$. Consider two cases.

{\it Case 1.} If $N = 2r - 1,$ then the coefficient at term $x_1 x_r \dx_1 \dx_r$ in $s_N(X_1, \ldots, X_N)$ is $E_G(\{ 1,1\})$ (sinks are $\{1,r \}$). There is only one 2-decomposition with such sources and sinks: the paths are $(1\to 1)$ and $(1 \to 2 \to 2 \to \cdots \to r-1 \to r \to r)$. Hence, $E_G(\{ 1,1\}) = |\text{Sh}(1, 2r-2)| = q(1,2r-2) = 1 > 0$ and $s_N \not\in A_{n}^{(1,1)}$.

{\it Case 2.} If $N = 2r$, then consider the term $x_1 x_2 \dx_2 \dx_r$ and its coefficient in $s_N$, which is $E_G(\{ 1,2\})$ (sinks are $\{ 2,r\}$). The possible 2-decompositions here are 

(1) $(1 \to 2)$ and $(2 \to 2 \to \cdots \to r-1 \to r-1 \to r)$ and

(2) $(1 \to 2 \to 2)$ and $(2 \to 3 \to \cdots \to r-1 \to r-1 \to r)$.

Therefore, $E_G(\{ 1,2\}) = q(1, 2r-1) + q(2,2r-2) = r - 1 > 0$ and $s_N \not\in A_{n}^{(1,1)}$.
\end{proof}

\section{Open questions}\label{open}

We propose several problems concerning the minimal identities in Weyl algebra. 

\

\noindent{\bf Problem 1.} What is $c = c(n)$ ($n > 3$) for which $s_{c} = 0$ is a minimal polynomial identity on $A_n^{(1,1)}$? We have shown that $2n < c \le n^2$. 

Using graph-theoretic interpretation, question becomes the following. What is relation between $|E|$ and $|V|$ such that digraph $G = (V, E)$ has $E_{G}(I) = 0$ for all sources $I$? This formulation implies from our graph-theoretic interpretation. For instance, in the classical Amitsur-Levitzki theorem we have $E_{G}(\{ i\}) = 0$ for all $i \in V$ if $|E| \ge 2|V|$.

Consider a more general setting. Recall that $A_n^{(p,p)} \subset {A}_n$ is the subspace of Weyl algebra generated as follows
$$
A_n^{(p,p)} := \langle x_{i_1} \cdots x_{j_p} \dx_{j_1} \cdots \dx_{j_p}\ |\ i_1, \ldots, i_p, j_1, \ldots, j_p \in [n] \rangle.
$$
What is $c(p,n)$ such that $s_{c(p,n)} = 0$ is a minimal identity on $A_n^{(p,p)}$? 

\

\noindent{\bf Problem 2.} Let $A_n^{*(p)} := \displaystyle\bigoplus_{i = 1}^p A_n^{(i,i)}$. What is a minimal identity on $A_n^{*(p)}$? For instance, $s_{2}$ is identity on $A_1^{*(p)}$, since
$$
x^{\ell_1} \dx^{\ell_1}  x^{\ell_2} \dx^{\ell_2} = \sum_{i \ge 0} i!\binom{\ell_1}{i} \binom{\ell_2}{i} x^{\ell_1 + \ell_2 - i} \dx^{\ell_1 + \ell_2 - i}
$$
and so
$
s_2(x^{\ell_1} \dx^{\ell_1}, x^{\ell_2} \dx^{\ell_2}) = 0.
$

\end{document}